\documentclass[a4paper,oneside,11pt]{amsart}

\reversemarginpar

\usepackage[colorlinks,hyperindex,linkcolor=blue,urlcolor=black,pdftitle={On power bounded operators with holomorphic eigenvectors, II},pdfauthor={Maria F.Gamal'}]
{hyperref}

\usepackage{amsmath}
\usepackage{amsfonts}
\usepackage{amssymb}
\usepackage{amsthm}

\usepackage[english]{babel}
\usepackage{enumerate}

\usepackage{graphicx}
\usepackage{color}

\usepackage[abbrev]{amsrefs}

\numberwithin{equation}{section}

\theoremstyle{plain}
\newtheorem{theorem}{Theorem}[section]
\newtheorem{lemma}[theorem]{Lemma}
\newtheorem{corollary}[theorem]{Corollary}

\theoremstyle{definition}
\newtheorem{remark}[theorem]{Remark}
\newtheorem{example}[theorem]{Example}

\begin{document}

\title[On power bounded operators]{On power bounded operators with holomorphic eigenvectors, II}
\author{Maria F. Gamal'}
\address{
 St. Petersburg Branch\\ V. A. Steklov Institute 
of Mathematics\\
 Russian Academy of Sciences\\ Fontanka 27, St. Petersburg\\ 
191023, Russia  
}
\email{gamal@pdmi.ras.ru}


\subjclass[2010]{Primary 47A05, 47B99, 47B32, 30H10}

\keywords{Power bounded operator,  unilateral shift, similarity, quasisimilarity, quasiaffine transform, analytic family of eigenvectors.}

\begin{abstract}
In \cite{uch} (among other results), M. Uchiyama gave the necessary and sufficient conditions 
for contractions 
to be  similar to the 
unilateral shift $S$ of multiplicity $1$  
in terms of norm-estimates of 
complete analytic families of eigenvectors of their  adjoints. 
In \cite{gam}, it was shown that this result 
for contractions 
can't be extended to power bounded operators. Namely, a cyclic  power bounded operator 
was constructed which has the requested norm-estimates, is a quasiaffine transform of $S$, but 
is not quasisimilar to $S$. In this paper, it is shown that the additional assumption on 
 a  power bounded operator to be quasisimilar to $S$ (with the requested norm-estimates) does not imply 
similarity to $S$. A question whether the criterion for contractions to be similar to $S$ can be generalized 
to polynomially bounded operators remains open.

Also, for every cardinal number $2\leq N\leq \infty$ a  power bounded operator $T$ is constructed such that $T$ 
 is a quasiaffine transform of $S$ and $\dim\ker T^*=N$. This is impossible for polynomially bounded operators. 
Moreover, the constructed operators $T$ have the requested norm-estimates of 
complete analytic families of eigenvectors
of $T^*$.    
  \end{abstract}

 \maketitle

\section{Introduction}

Let $\mathcal H$ be a (complex, separable) Hilbert space, 
and let $T$ be 
a (linear, bounded) operator acting 
on $\mathcal H$. The operator $T$ is called \emph{power bounded}, 
if 
$\sup_{n\geq 1}\|T^n\| < \infty$. The operator $T$ is called \emph{polynomially bounded}, if 
there exists a constant $C>0$ such that $\|p(T)\|\leq C\max\{|p(z)|, |z|\leq 1\}$ for every (analytic) polinomial $p$. 
The operator $T$ is a \emph{contraction} if $\|T\|\leq 1$. 
 It is well known that a contraction is polynomially bounded, and, consequently, power bounded. 

Let $T$ and $R$ be operators on spaces $\mathcal H$ and $\mathcal K$, 
respectively, 
and let 
$X:\mathcal H\to\mathcal K$ be a (linear, bounded) transformation 
such that 
$X$ \emph{intertwines} $T$ and $R$, 
that is, $XT=RX$. If $X$ is unitary, then $T$ and $R$ 
are called \emph{unitarily equivalent}, 
in notation: 
$T\cong R$. If $X$ is invertible (that is, its inverse is \emph{bounded}),
 then $T$ and $R$ 
are called \emph{similar}, 
in notation: $T\approx R$.
 If $X$ is a \emph{quasiaffinity}, that is, $\ker X=\{0\}$ and 
$\operatorname{clos}X\mathcal H=\mathcal K$, 
then 
$T$ is called a \emph{quasiaffine transform} of $R$, 
in notation: $T\prec R$. If $T\prec R$ and 
$R\prec T$, 
 then $T$ and $R$ are called \emph{quasisimilar}, 
in notation: $T\sim R$. 

In \cite{uch}, necessary and sufficient conditions 
for contractions to be quasiaffine transforms,  
quasisimilar, or similar to 
unilateral shifts of finite multiplicity in terms 
of norm-estimates  of 
complete analytic families of eigenvectors 
of their adjoints are given. In \cite{gam}, 
the result from  \cite{uch} for contractions to be quasiaffine transforms of
 unilateral shifts 
of finite multiplicity 
is generalized to power bounded operators.  
Also, in \cite{gam} an example  of a cyclic  power bounded operator $T_0$ is given such that $T_0$  satisfies  sufficient conditions 
on contractions to be similar to the unilateral shift $S$ of multiplicity $1$, 
but  \emph{there is no} contraction $R$ such that $R\prec T$. 
This example is based on some example from \cite{mt}.  

In this paper, it is shown that the additional assumption $T\sim S$ on 
 a  power bounded operator $T$ (with the requested norm-estimates of 
complete analytic family of eigenvectors
of $T^*$) also \emph{does not} imply the similarity  $T\approx S$. The constructed operator $T$ has the following property: 
there exist two invariant subspaces  $\mathcal M_1$, $\mathcal M_2$ of $T$ such that $T|_{\mathcal M_k}\approx S$ ($k=1,2$) and 
$\mathcal M_1\vee\mathcal M_2$ is the whole space on which $T$ acts. The same property takes place for polynomially bounded operators 
that are quasiaffine transforms of $S$ \cite{gam3}. 
A question whether the criterion for contractions to be similar to $S$ can be generalized 
to polynomially bounded operators remains open.

If a polynomially bounded operator $T$ is such that $T\prec S$, then the range of $T$ is closed and  $\dim\ker T^*=1$ (\cite{tak}, 
\cite{berpr}, \cite{gam1}, see Remark \ref{33} below).  The range of a cyclic  power bounded operator $T_0$ constructed in \cite{gam} 
is not closed. It allows, for every cardinal number $2\leq N\leq \infty$,  to construct example  of a   power bounded operator $T$ 
such that $T\prec S$, a 
complete analytic family of eigenvectors of $T^*$ has the requested norm-estimates, and $\dim\ker T^*=N$. 

The paper is organized as follows. 
In Sec. 2 a  power bounded operator $T$ is constructed such that $T\sim S$,
 a complete analytic family of eigenvectors of $T^*$ has the requested norm-estimates, but $T\not\approx S$. 
In Sec. 3 it is shown that if a  power bounded operator $T$ is  such that  $T\prec S$, then  $\dim\ker T^*$ 
can be arbitrarily large. In Sec. 4 it is shown that for $T_0$ constructed in \cite{gam} $\sigma_e(T_0)$ is the closed unit disc. 

The following notation will be used. Let $\mathcal H$ be a Hilbert space, and let $\mathcal M$ be its
(linear, closed)  subspace.
By $I_{\mathcal H}$ and $P_{\mathcal M}$ the identical operator on $\mathcal H$ and 
the orthogonal projection from $\mathcal H$ onto $\mathcal M$ are denoted, respectively. 
For an operator  $T\colon \mathcal H\to \mathcal H$, a subspace $\mathcal M$  of $\mathcal H$ is called 
\emph{invariant subspace} of $T$, if $T\mathcal M\subset\mathcal M$. The complete lattice of all invariant 
 subspaces of $T$ is denoted by  $\operatorname{Lat}T$. 

\emph{The words} ``operator" \emph{and} ``transformation"  \emph{mean that the linear mappings under consideration are bounded.} 


The symbols $\mathbb D$, $\mathbb T$, and $m$ denote the open unit disc, the unit circle, and the normalized Lebesgue 
measure on  $\mathbb T$, respectively. $H^\infty$ is the Banach algebra of all analytic bounded functions in $\mathbb D$. $H^2$ is the Hardy space on  $\mathbb D$, 
$S\colon H^2\to H^2$, $(Sh)(z) = zh(z)$ ($z\in\mathbb D, h\in H^2$)  
is the unilateral shift of multiplicity 1. 
Set \begin{equation}\label{hlambda} h_\lambda(z)=\frac{1}{1-\overline\lambda z}, \ \ \ \ z,\lambda\in\mathbb D. \end{equation}
It is well known and easy to see that  
$$ S^*h_\lambda=\overline\lambda h_\lambda, \ \ \ \|h_\lambda\|=\frac{1}{(1-|\lambda|^2)^{1/2}} \ \ \ \ \ \ (\lambda\in\mathbb D), $$
and the function $ \mathbb D\to H^2$, $\lambda\mapsto  h_\lambda$, is  conjugate analytic. 

If $T$ is an operator on a Hilbert space $\mathcal H$ and a quasiaffinity $X$ is such that $XT=SX$, then 
$$ T^*X^*h_\lambda=\overline\lambda X^* h_\lambda, \ \ \ \|X^*h_\lambda\|\leq \|X\|\frac{1}{(1-|\lambda|^2)^{1/2}} 
\ \ \ \ \ \ (\lambda\in\mathbb D), $$
and the function $ \mathbb D\to \mathcal H$, $\lambda\mapsto  X^* h_\lambda$, is  conjugate analytic. 
If, in addition, $T$ is a contraction and $\|X^*h_\lambda\|\asymp\frac{1}{(1-|\lambda|^2)^{1/2}}$, then, by \cite{uch}, 
$T\approx S$.

\section{Example of operator quasisimilar to $S$}

In this section  a  power bounded operator $T$ is constructed such that $T\sim S$,
 a complete analytic family of eigenvectors $k_\lambda$ of $T^*$ has the estimate $\|k_\lambda\|\asymp\|h_\lambda\|$,   but $T$ is not polynomially bounded. Consequently, $T\not\approx S$. 

For an inner function $\theta$ set $\mathcal K_\theta=H^2\ominus\theta H^2$, 
 $T_\theta=P_{\mathcal K_\theta}S|_{\mathcal K_\theta}$, and 
\begin{equation}\label{thetalam}h_{\theta,\lambda}(z)=\frac{1-\overline{\theta(\lambda)}\theta(z)}{1-\overline\lambda z}, \ \ \ z,\lambda\in\mathbb D.\end{equation} 
It is well known and easy to see that  
\begin{equation}\label{thetalambda} P_{\theta H^2}h_\lambda=\overline{\theta(\lambda)}\theta h_\lambda
 \ \ \text{and } \ \ P_{\mathcal K_\theta}h_\lambda= h_{\theta,\lambda}. \end{equation}
Furthermore,   
\begin{equation}\label{flambda} f(\lambda)=(f,h_{\theta,\lambda}) \ \text{ for every }  \lambda\in\mathbb D 
 \text{  and }f\in\mathcal K_\theta. \end{equation}

\subsection{Bases of a Hilbert space}

In this subsection we recall the notions and properties of (not orthogonal) bases of a Hilbert space. For references, see 
{\cite[Ch. VI.3]{nik86}} or {\cite[Ch. I.A.5.1, I.A.5.6.2, II.C.3.1]{nik02}}. 

Let $\mathcal H$ be a Hilbert space, and let  $\{x_n\}_{n\geq 0}\subset\mathcal H$. The family  $\{x_n\}_{n\geq 0}$ is called 
an \emph{unconditional basis of $\mathcal H$}, if for every $x\in\mathcal H$ there exists a family  $\{a_n\}_{n\geq 0}\subset \mathbb C$ 
such that $x=\sum_{n\geq 0}a_nx_n$ and the series  $\sum_{n\geq 0}a_nx_n$ converges unconditionally, that is, for every 
$\varepsilon>0$ there exists a finite $\frak N\subset\{0,1,\ldots\}$ such that $\|x-\sum_{n\in\frak N'}a_nx_n\|<\varepsilon$ 
for every finite $\frak N\subset\frak N'\subset\{0,1,\ldots\}$. The family  $\{x_n\}_{n\geq 0}$ is called 
a \emph{Riesz basis of $\mathcal H$}, if the mapping $W$ acting by the formula $We_n=x_n$ for an orthonormal basis 
 $\{e_n\}_{n\geq 0}$ is an invertible transformation. Let $\{x_n\}_{n\geq 0}$ is such that $\mathcal H=\vee_{n\geq 0}x_n$, 
$x_n\not\in\vee_{k\neq n}x_k$ for all $n$,  and $\|x_n\|\asymp 1$. Then $\{x_n\}_{n\geq 0}$ is an unconditional basis 
if and only if $\{x_n\}_{n\geq 0}$ is a Riesz basis. 

Let $\mathcal H=\vee_{n\geq 0}x_n$, and let $\{x_n\}_{n= 0}^N$ are linear independent for every finite $N$. Define mappings
$\mathcal Q_n$ and $\mathcal P_n$ on the linear set 
$$\Bigl\{\sum_{n=0}^N a_n x_n, \{a_n\}_{n\geq 0}\subset \mathbb C, N=0,1,\ldots\Bigr\}$$ by the formulas
 \begin{equation}\label{qqn}\mathcal Q_n\sum_{k\geq 0} a_k x_k=a_n x_n \end{equation}
and 
\begin{equation}\label{ppndef}\mathcal P_n\sum_{k\geq 0} a_k x_k=\sum_{k=0}^n a_k x_k.
\end{equation} 
Clearly, $\mathcal Q_n=\mathcal P_n-\mathcal P_{n-1}$ for all $n\geq 1$, and $\mathcal Q_0=\mathcal P_0$.

Let $n\geq 0$. Then the following are equivalent: (i) $\mathcal Q_n$ can be extended on $\mathcal H$ as an operator; 
(ii) there exists $x'_n\in\mathcal H$ such that  $(x'_n,x_n)=1$ and  $(x'_n,x_k)=0$, if $k\neq n$; (iii) $x_n\not\in\vee_{k\neq n}x_k$. 

If $\mathcal P_n$ can be extended on $\mathcal H$ as operators for all $n$ and  
 \begin{equation}\label{ppnest}\sup_{n\geq 0}\|\mathcal P_n\|<\infty, \end{equation}
then $ \mathcal P_nx=\sum_{k= 0}^n(x,x'_k)x_k $ for all $n$ and
 $$x=\lim_n\mathcal P_nx=\sum_{k\geq 0}(x,x'_k)x_k\Bigl(=\lim_n\sum_{k= 0}^n(x,x'_k)x_k\Bigr) \ \ \text{ for every } x\in\mathcal H.$$

\begin{lemma}\label{quasi} Suppose that $\mathcal H$ and $\mathcal K$ are Hilbert spaces, $\mathcal H=\vee_{n\geq 0}\ x_n$, 
$\mathcal K=\vee_{n\geq 0}\ y_n$, $\mathcal Q_n$ and $\mathcal P_n$  defined by \eqref{qqn} and \eqref{ppndef} for 
 $\{y_n\}_{n\geq 0}$ and $\{x_n\}_{n\geq 0}$, respectively, are operators for all $n$, and \eqref{ppnest} is fulfilled for $\mathcal P_n$. 
Furthermore, suppose that the family $\{c_n\}_{n\geq 0}\subset \mathbb C$ is such that $c_n\neq 0$ for every $n\geq 0$ and the mapping $X$ acting by the formula $Xx_n=c_ny_n$,  $n\geq 0$,  is a transformation. Then $X\colon\mathcal H\to\mathcal K$ is a quasiaffinity. 
\end{lemma}
\begin{proof} Since $y_n\in X\mathcal H$ and $\mathcal K=\vee_{n\geq 0}\ y_n$, we conclude that $\operatorname{clos}X\mathcal H=\mathcal K$.
Let $x\in\mathcal H$ be such that $Xx=0$. We have $x=\lim_N\sum_{k= 0}^N a_kx_k$ (where $a_k=(x,x'_k)$), 
therefore,  
$$0=\mathcal Q_n Xx =\mathcal Q_n X\Bigl(\lim_N\sum_{k= 0}^N a_kx_k\Bigr)=\lim_N\mathcal Q_n \sum_{k= 0}^N a_kc_ky_k= a_nc_ny_n $$
for every $n\geq 0$. Since $c_n\neq 0$, we conclude that $a_n=0$ for every $n\geq 0$. Therefore, $x=0$.
\end{proof}

\begin{lemma}\label{lemgzw}Suppose that $\mathcal H$ and   $\mathcal K$ are Hilbert spaces, $\{u_n\}_{n\geq 0}$ is a Riesz basis of  $\mathcal K$, the family  $\{x_n\}_{n\geq 0}\subset \mathcal H$ is such that
\begin{equation} \label{hh0xn}\mathcal H=\bigvee_{n\geq 0} x_n, \ \ \  \inf_{n\geq 0}\|x_n\|>0, \end{equation} 
and $\mathcal Q_n$ defined by \eqref{qqn} are such that
\begin{equation} \label{qnest} \sup_{n\geq 0}\|\mathcal Q_n\|<\infty. \end{equation} 
 Let $\{c_n\}_{n\geq 0}\subset \mathbb C$ be such that 
$\sum_{n\geq 0}|c_n|^2<\infty$.  Then  the mapping $Z\colon  \mathcal H  \to \mathcal K$  acting by the formula 
$Z x_n = c_n u_n$, $n\geq 0$,  is  a transformation.\end{lemma}
\begin{proof} Since $\{u_n\}_{n\geq 0}$ is a Riesz basis, 
\begin{align*}\Bigl\|Z\sum_{n\geq 0}a_n x_n\Bigr\|^2 & =\Bigl\|\sum_{n\geq 0}c_n a_n u_n\Bigr\|^2\asymp\sum_{n\geq 0}|c_n|^2|a_n|^2\\&\leq
\frac{1}{(\inf_{n\geq 0}\|x_n\|)^2}\sum_{n\geq 0}|c_n|^2\Bigl\|\mathcal Q_n \sum_{k\geq 0}a_k x_k\Bigr\|^2\\&
\leq \frac{(\sup_{n\geq 0}\|\mathcal Q_n\|)^2}{(\inf_{n\geq 0}\|x_n\|)^2}\sum_{n\geq 0}|c_n|^2\Bigl\|\sum_{k\geq 0}a_k x_k\Bigr\|^2. \end{align*}
\end{proof}

 In Lemma \ref{lempsi} the notion of a Helson--Szeg\"o weight function is used. To the definition of this notion we refer to  the references in Lemma \ref{lempsi}. This notion will not be used in the sequel. 

\begin{lemma}[{\cite[Theorem 2.1]{borsp}}, {\cite[Ch. VIII.6]{nik86}}, {\cite[Lemma I.A.5.2.5, Theorem I.A.5.4.1]{nik02}}] 
\label{lempsi}
Suppose that $\psi\in H^\infty$ is an outer function, $|\psi|^2$ is a Helson--Szeg\"o weight function, and $1/\psi\not\in H^\infty$.  
Set $\mathcal H_0=H^2$ and $x_n=\chi^n\psi$, $n\geq 0$, where $\chi(\zeta)=\zeta$, $\zeta\in\mathbb T$. 
 Then  $\mathcal H_0$ and $\{x_n\}_{n\geq 0}$ satisfy \eqref{hh0xn}, 
$\{\mathcal P_n\}_{n\geq 0}$ defined by \eqref{ppndef} satisfies \eqref{ppnest}, 
for every Riesz basis $\{u_n\}_{n\geq 0}$ of a Hilbert space $\mathcal K$ the mapping 
$W\colon\mathcal K\to \mathcal H_0$ acting by the formula $Wu_n=x_n$, $n\geq 0$, is a transformation, 
but $\{x_n\}_{n\geq 0}$ is not a Riesz basis of  $\mathcal H_0$.
\end{lemma}
\begin{proof} Clearly, $\{\chi^n\}_{n\geq 0}$  is an orthonormal basis of $H^2$. 
Define $J\colon H^2\to H^2$ by the formula $Jh=\psi h$, $h\in H^2$. Clearly,  $J\chi^n= x_n$, $n\geq 0$, 
 and $J$ is an operator, because $\psi\in H^\infty$. Thus, the statement about $W$ is proved. 

If we assume that 
$\{x_n\}_{n\geq 0}$ is a Riesz basis of $H^2$,  then the operator  $J$ must have bounded inverse on $H^2$. 
Since $1/\psi\not\in H^\infty$, we conclude that $J$ has no bounded inverse. 

All remaining statements follow from the references.
 \end{proof}

\begin{example}[{\cite[Example 3.3.2]{borsp}}, {\cite[Ch. I.A.5.5]{nik02}}]\label{exapsi} Let $0<\alpha<1/2$. Set $\psi(z)=(1-z)^\alpha$, $z\in\mathbb D$. 
 Then $\psi$ satisfies  Lemma \ref{lempsi}.
\end{example}

Recall that an operator $T$ satisfies the Tadmor--Ritt condition, 
if there exists $C>0$ such that $\|(T-zI)^{-1}\|\leq C/|z-1|$ for $z\in\mathbb C$, $|z|>1$. The Tadmor--Ritt condition implies 
power boundedness \cite{lyu}, \cite{naze}, \cite{vitse}. The Tadmor--Ritt condition will not be used in the sequel. 

\begin{lemma}[{\cite[Lemma 2.2]{vitsejfa}}]\label{lemvitse} Suppose that $\{\lambda_n\}_{n\geq 0}\subset (0,1)$, $\lambda_n< \lambda_{n+1}$ for all $n$, 
  and $\lambda_n\to 1$. Furthermore, suppose that $\mathcal H_0$ is a Hilbert space, 
 $\mathcal H_0=\vee_{n\geq 0}x_n$, 
 $\mathcal P_n\colon\mathcal H_0\to\mathcal H_0$ defined by \eqref{ppndef} are operators, and \eqref{ppnest} is fulfilled. 
Then the mapping $R_0\colon\mathcal H_0\to \mathcal H_0$ acting by the formula 
$$R_0 x_n=\lambda_n x_n,\ \ \ n\geq 0,$$ is an operator,  
 $R_0$ satisfies the Tadmor--Ritt condition and, consequently, is power bounded. 
\end{lemma}

\subsection{Blaschke product}

In this subsection we recall the well-known facts about Blaschke products which will be used in the sequel. For references, see 
{\cite[Ch. VI.2,  IX.3]{nik86}} or {\cite[Ch. II.C.3.2, Lemma II.C.3.2.18]{nik02}}.

For $\lambda\in\mathbb D$, a Blaschke factor is $b_\lambda(z)=\frac{|\lambda|}{\lambda}\frac{\lambda-z}{1-\overline\lambda z}$, 
$z\in\mathbb D$. The following equality will be used:
\begin{equation}\label{blmod} 1-|b_\lambda(z)|^2=\frac{(1-|z|^2)(1-|\lambda|^2)}{|1-\overline\lambda z|^2}. \end{equation}
If $\{\lambda_n\}_n\subset\mathbb D$ satisfies  the Blaschke condition $\sum_n(1-|\lambda_n|)<\infty$, 
then the Blaschke product  $B=\prod_n b_{\lambda_n}$ converges and $B$ is an inner function. 

Let  $B=\prod_n b_{\lambda_n}$ be a Blaschke product with simple zeros, that is,  $\lambda_n\neq \lambda_k$, if $n\neq k$. 
Set $B_n=\prod_{k\neq n} b_{\lambda_k}$, 
\begin{equation}\label{uv}u_n=(1-|\lambda_n|^2)^{1/2}h_{\lambda_n}, \  \ \ v_n=\frac{1}{B_n(\lambda_n)}B_n u_n, \end{equation}
where $h_\lambda$ are defined in \eqref{hlambda}, 
then
$$\|u_n\|=1, \ \|v_n\|=\frac{1}{|B_n(\lambda_n)|}, \ \ \  (v_n, u_n)=1,  \ (v_n,u_k)=0, \ \text{   if  } n\neq k, $$
$$\mathcal K_B=\vee_n u_n=\vee_n v_n, \  \ \text{   and  } \ \ T_B^* u_n=\overline{\lambda_n} u_n. $$

Let $\{\lambda_n\}_n\subset\mathbb D$ be such that  $\lambda_n\neq \lambda_k$, if $n\neq k$. 
The family $\{\lambda_n\}_n$ satisfies  the \emph{Carleson interpolating condition} 
(the \emph{Carleson condition} for brevity), if 
\begin{equation}\label{carl} \text{there exists }\delta>0  \text{ such that }|B_n(\lambda_n)|\geq \delta 
\ \text{  for every } n. \end{equation}
Then $\{u_n\}_n$ and $\{v_n\}_n$ are Riesz bases of $\mathcal K_B$, and 
\begin{equation}\label{fuv} f=\sum_n(f,v_n)u_n \  \ \text{  for every } f\in\mathcal K_B. \end{equation}
Set
\begin{equation}\label{dn} D_n=\{z\in\mathbb D\ :\ |b_{\lambda_n}(z)|\leq \delta/3\}. \end{equation}
Then
\begin{equation}\label{bbnddn}  |B_n(z)|\geq\delta/2,  \  \ \text{  if } z\in D_n,  \end{equation}
and
\begin{equation}\label{bbnotdd} |B(z)|\geq \delta^2/6, \  \ \text{  if } z\in\mathbb D\setminus\cup_n D_n. \end{equation}

\begin{lemma}\label{lem11} Suppose that $\{\lambda_n\}_{n\geq 0}\subset\mathbb D$, $\lambda_n\neq \lambda_k$,  if $n\neq k$, and 
 $\{\lambda_n\}_{n\geq 0}$ satisfies  the Carleson condition \eqref{carl}. Furthermore,  suppose that $\mathcal H_0$ is a Hilbert space, 
 $\{x_n\}_{n\geq 0}\subset\mathcal H_0$ satisfies \eqref{hh0xn}, 
 $\mathcal Q_n\colon\mathcal H_0\to\mathcal H_0$ acting by the formula \eqref{qqn} are operators and 
\eqref{qnest} is fulfilled. 
Finally, suppose that $$W\colon\mathcal K_B\to\mathcal H_0 \text{  is a transformation and } x_n=W u_n \text{  for all } n\geq 0,$$ 
where $u_n$ are defined in \eqref{uv}. 
Then $$\|Wh_{B,\lambda}\|\geq \frac{\delta}{2}\Bigl(1-\frac{\delta^2}{9}\Bigr)^{1/2}
\frac{\inf_{n\geq 0}\|x_n\|}{\sup_{n\geq 0}\|\mathcal Q_n\|}\frac{1}{(1-|\lambda|^2)^{1/2}}\  \text{ for every } \lambda\in\cup_{n\geq 0} D_n,$$
where $h_{B,\lambda}$ are defined in \eqref{thetalam}, $\delta$ is from \eqref{carl} and $D_n$ are defined in \eqref{dn}.
\end{lemma}
\begin{proof}
 By \eqref{fuv}, 
$h_{B,\lambda}=\sum_{k\geq 0}(h_{B,\lambda},v_k)u_k$. Therefore, $$W h_{B,\lambda}=\sum_{k\geq 0}(h_{B,\lambda},v_k)x_k.$$ We have 
$$\|W h_{B,\lambda}\|\geq\frac{1}{\sup_k\|\mathcal Q_k\|}\|\mathcal Q_n W h_{B,\lambda}\|
=\frac{1}{\sup_k\|\mathcal Q_k\|}|(h_{B,\lambda},v_n)|\|x_n\|.$$
By \eqref{flambda}, \eqref{uv}, and  \eqref{bbnddn}, 
\begin{align*}|(h_{B,\lambda},v_n)| & =|v_n(\lambda)|=\frac{|B_n(\lambda)|}{|B_n(\lambda_n)|}\frac{(1-|\lambda_n|^2)^{1/2}}{|1-\overline\lambda_n\lambda|}\\&
\geq \frac{\delta}{2}\frac{(1-|\lambda|^2)^{1/2}(1-|\lambda_n|^2)^{1/2}}{|1-\overline\lambda_n\lambda|}\frac{1}{(1-|\lambda|^2)^{1/2}}, 
\ \  \lambda\in D_n.\end{align*}
By \eqref{blmod} and \eqref{dn}, $$\frac{(1-|\lambda|^2)^{1/2}(1-|\lambda_n|^2)^{1/2}}{|1-\overline\lambda_n\lambda|}=
(1-|b_{\lambda_n}(\lambda)|^2)^{1/2}\geq\Bigl(1-\frac{\delta^2}{9}\Bigr)^{1/2},\ \  \lambda\in D_n.$$
\end{proof}

\subsection{Construction of example}

The following lemma is a corollary of Sec. 2.1 and 2.2. 

\begin{lemma}\label{corg} Suppose that  $\{\lambda_n\}_{n\geq 0}$, $\mathcal H_0$ and $\{x_n\}_{n\geq 0}$ satisfy assumption of Lemma \ref{lemvitse}, 
 $\{\lambda_n\}_{n\geq 0}$ satisfies the Carleson  condition \eqref{carl}, and $\{x_n\}_{n\geq 0}$ satisfies \eqref{hh0xn}.  Let $R_0$ be the operator from Lemma \ref{lemvitse}. 
Let $g\in H^\infty$ be such that 
\begin{equation}\label{glambda}\sum_{n\geq 0}|g(\lambda_n)|^2<\infty.\end{equation} Set $c_n=g(\lambda_n)$, $n\geq 0$. 
Define $Z$ as in Lemma \ref{lemgzw}.  
Suppose that the mapping $W\colon \mathcal K_B\to \mathcal H_0$ acting  by the formula $W u_n=x_n$, $n\geq 0$,  is  a transformation.
 Then  $W T_B^\ast=R_0 W$,  $ZR_0=T_B^\ast Z$, and $ZW=g(T_B^\ast)$. 
\end{lemma}
\begin{proof}Set  $\mathcal Q_n=\mathcal P_n-\mathcal P_{n-1}$, $n\geq 1$, and $\mathcal Q_0=\mathcal P_0$. 
 Clearly, $\mathcal Q_n$ act by the formula \eqref{qqn}, 
 and  \eqref{qnest} is fulfilled. The lemma follows from Lemma \ref{lemgzw} and the definitions of $W$, $Z$, and $R_0$. \end{proof}

\begin{example} Suppose that  $\{\lambda_n\}_{n\geq 0}\subset(0,1)$ satisfies the Blaschke condition, $\lambda_n\neq\lambda_k$, 
if $n\neq k$,  and $g(z)=(1-z)^{1/2}$, $z\in\mathbb D$. 
Then $g$ is outer, $g$ satisfies \eqref{glambda} and $\sum_{n\geq 0}|\widehat g(n)|<\infty$ by {\cite[Corollary of Theorem 3.15]{duren}}, because 
$g'\in H^1$.
\end{example}

The following theorem is the main result of Sec. 2.

\begin{theorem}\label{thm18}  Suppose that $\{\lambda_n\}_{n\geq 0}\subset(0,1)$,  $\lambda_n< \lambda_{n+1}$ for all $n$,  and 
 $\{\lambda_n\}_{n\geq 0}$ satisfies the Carleson  condition \eqref{carl}. Suppose that 
$\mathcal H_0$ and $\{x_n\}_{n\geq 0}$ satisfy the assumptions of Lemmas \ref{lemvitse} and \ref{lem11}, 
but $\{x_n\}_{n\geq 0}$ is not a Riezs basis of  $\mathcal H_0$. Suppose that $g\in H^\infty$ is outer, and  $g$ satisfies \eqref{glambda}. 
Set $g_*(z)=\overline{g(\overline z)}$, $z\in\mathbb D$. 
Finally, suppose that  $R_0$, $W$, $Z$ are from  Lemmas \ref{lemvitse}, \ref{lem11}, and  \ref{corg}, 
respectively. 
Set $$T\colon BH^2\oplus \mathcal H_0\to BH^2\oplus \mathcal H_0, \ Y\colon BH^2\oplus \mathcal H_0\to H^2,\  X\colon H^2\to BH^2\oplus \mathcal H_0,$$ 
$$T=\begin{pmatrix}S|_{BH^2} & P_{BH^2}S|_{\mathcal K_B}W^\ast \\ \mathbb O & R_0^\ast\end{pmatrix},$$
$$ 
Y=\begin{pmatrix}I_{BH^2} & \mathbb O \\ \mathbb O & W^\ast\end{pmatrix}, 
\ \ \ X=\begin{pmatrix}g_\ast(S)|_{BH^2} & P_{BH^2}g_\ast(S)|_{\mathcal K_B} \\ \mathbb O & Z^\ast\end{pmatrix}.$$
Then $T$ is power bounded, $T$ is not polynomially bounded, $X$ and $Y$ are quasiaffinities, $YT=SY$, $XS=TX$, $YX=g_\ast(S)$, and 
if $\sum_{n\geq 0}|\widehat g(n)|<\infty$, then $XY=g_*(T)$. Moreover, 
\begin{equation}\label{estlambda}\|Y^\ast h_\lambda\|\asymp\|h_\lambda\|=\frac{1}{(1-|\lambda|^2)^{1/2}} 
\ \ \  (\lambda\in\mathbb D),\end{equation}
where $h_\lambda$ are defined in \eqref{hlambda}. 
\end{theorem}

\begin{proof} By Lemma \ref{quasi}, $W$ and $Z$ are quasiaffinities. Therefore,  $X$ and $Y$ are quasiaffinities. 
Intertwining properties of $X$ and $Y$  and the equality  $YX=g_\ast(S)$ easy follow from the  construction of $T$, $X$, $Y$ 
and Lemma \ref{corg}. 

The power boundedness  of $T$ follows from the power boundedness of $R_0$ and the equality 
$$ P_{BH^2}T^n|_{\mathcal H_0}=P_{BH^2}S^n|_{\mathcal K_B}W^\ast,$$
which is a consequence of the relation $YT=SY$. 

If  $\sum_{n\geq 0}|\widehat g(n)|<\infty$, then $g_*(T) = \sum_{n\geq 0}\overline{\widehat g(n)}T^n$. 
Since $YXY=g_*(S)Y=Yg_*(T)$ and $\ker Y$=\{0\}, 
we conclude that $XY=g_*(T)$. 
 
It is well known that if $R_0$ is polynomially bounded, then  $\{x_n\}_{n\geq 0}$ is an unconditional basis of $\mathcal H_0$
(because  $\{\lambda_n\}_{n\geq 0}$ satisfies the Carleson condition \eqref{carl}, see, for example, {\cite[Lemma 2.3]{vitsejfa}}), 
a contradiction to assumption. Since $R_0$ is not polynomially bounded, we conclude that $T$ is not polynomially bounded. 

Clearly, $\|Y^\ast h_\lambda\|\leq\|Y\|\|h_\lambda\|$ for every $\lambda\in\mathbb D$. Therefore, 
to prove \eqref{estlambda}, it is sufficient to prove that there exists $c>0$ such that 
\begin{equation}\label{geqlambda}\|Y^\ast h_\lambda\|\geq c\frac{1}{(1-|\lambda|^2)^{1/2}} 
\ \ \text{for every } \lambda\in\mathbb D.\end{equation}
By \eqref{thetalambda},  $$\|Y^\ast h_\lambda\|^2=\|P_{BH^2}h_\lambda\|^2+\|WP_{\mathcal K_B}h_\lambda\|^2=
|B(\lambda)|^2\|h_\lambda\|^2+\|Wh_{B,\lambda}\|^2.$$
Let $\{D_n\}_{n\geq 0}$ be defined as in Lemma \ref{lem11}. By Lemma \ref{lem11}, there exists $c>0$ such that 
$$\|W h_{B,\lambda}\|\geq \frac{c}{(1-|\lambda|^2)^{1/2}}\ \ \text{ for every }\lambda\in\cup_{n\geq 0} D_n.$$ 
Let $\lambda\in\mathbb D\setminus\cup_n D_n$. By \eqref{bbnotdd},  $|B(\lambda)|\geq \delta^2/6$ for such $\lambda$. 
Therefore, $$\|Y^\ast h_\lambda\|\geq\frac{\delta^2}{6}\|h_\lambda\| \ \ \text{ for every } \lambda\in\mathbb D\setminus\cup_n D_n.$$ 
Estimate \eqref{geqlambda} is proved. 
\end{proof}

\begin{remark} If in  Theorem \ref{thm18} one assume that $\{x_n\}_{n\geq 0}$ \emph{ is}  a Riezs basis
instead of \emph{is not}, then $R_0$ is similar to a contraction. Therefore, $T$  is similar to a contraction 
by {\cite[Corollary 4.2]{cas}} and $T\approx S$ by {\cite[Theorem 3.8]{uch}}. \end{remark}

\begin{remark} A key step in the construction of the example is the existence of a family  $\{x_n\}_{n\geq 0}$ satifying \eqref{hh0xn}, 
such that the mapping $W$ acting by the formula $We_n=x_n$ ($n\geq 0$)  for an orthonormal basis  $\{e_n\}_{n\geq 0}$ is a  transformation, 
but $\{x_n\}_{n\geq 0}$ is not a Riesz basis (Lemma \ref{lempsi} and Example \ref{exapsi}). 
It seems the basis constructed in {\cite[Example III.14.5, p. 429]{singer}} does not have this property.
\end{remark}

\subsection{Existence of another shift-type invariant subspace}

Let $T$ be a polynomially bounded operator on a Hilbert space $\mathcal H$ such that $T\prec S$, 
and let $\mathcal N\in\operatorname{Lat}T$ be such that $T|_{\mathcal N}\approx S$. Then there exists   
 $\mathcal M\in\operatorname{Lat}T$  such that $T|_{\mathcal M}\approx S$ and 
$\mathcal N \vee\mathcal M =\mathcal H$ (see {\cite[Theorem 2.10]{gam3}}). The power bounded operator $T$ on the space $ BH^2\oplus \mathcal H_0$ constructed in Theorem \ref{thm18} has the invariant subspace $BH^2\oplus\{0\}$ 
such that $T|_{BH^2\oplus\{0\}}\cong S$. In this subsection we show that  there exists   
 $\mathcal M\in\operatorname{Lat}T$  such that $T|_{\mathcal M}\approx S$ and 
$(BH^2\oplus\{0\}) \vee\mathcal M = BH^2\oplus \mathcal H_0$ (although $T$ is not  polynomially bounded).

\begin{lemma}\label{lembertak} Suppose that $\mathcal H$ is a Hilbert space, $A\colon\mathcal H\to\mathcal H$ is an operator, and $x\in\mathcal H$. 
Then there exist an a.c. contraction $R\colon\mathcal K\to\mathcal K$ and a transformation $X\colon\mathcal K\to\mathcal H$ such that 
$XR=AX$ and $x\in X\mathcal K$ if and only if  
\begin{equation}\label{bertak3}\begin{gathered}\text{ there exists a function } w\in L^1(\mathbb T, m) \text{ such that } w\geq 0 \\
\text{ and } \ \|p(A)x\|^2\leq\int_{\mathbb T}|p|^2w\text{\rm d}m \ \text{ for every (analytic) polynomial }p.
\end{gathered}
\end{equation}
\end{lemma}

\begin{proof} ``If" part. Denote by $P^2(wm)$ the closure of (analytic) polynomials in $L^2(wm)$ and 
by $S_{wm}$ the operator of multiplication by the independent variable on $P^2(wm)$. 
Clearly, $S_{wm}$ is an a.c. contraction. Define the transformation  
$X\colon P^2(wm)\to\mathcal H$ by the formula $Xp=p(A)x$. Then $X1=x$ and $XS_{wm}=AX$.

``Only if" part. By assumption, there exists $u\in\mathcal K$ such that $x=Xu$. By {\cite[Lemma 3]{bertak}}, 
\eqref{bertak3} is fulfilled for $R$ and $u$ with some function $w$. 
Clearly,  \eqref{bertak3} is fulfilled for $A$ and $x$ with the function $\|X\|^2w$. 
\end{proof}

\begin{lemma}\label{lemshift} Suppose that $\theta$ is an inner function, $\mathcal H_0$ is a Hilbert space, 
$A_0\colon\mathcal H_0\to\mathcal H_0$ is an operator,  $Y_0\colon\mathcal H_0\to\mathcal K_\theta$ is a transformation, and 
$Y_0A_0=T_\theta Y_0$. 
 Furthermore, suppose that there exists $x\in\mathcal H_0$  such that 
$x$ is a cyclic vector for $A_0$ and \eqref{bertak3} is fulfilled for  $A_0$ and $x$. 
Set 
$$A\colon \theta H^2\oplus \mathcal H_0\to\theta H^2\oplus \mathcal H_0, \ \ \ 
A=\begin{pmatrix}S|_{\theta H^2} & P_{\theta H^2}S|_{\mathcal K_\theta}Y_0 \\ \mathbb O & A_0\end{pmatrix}.$$ 
Then there exists $\mathcal M\in\operatorname{Lat}A$ such that $A|_{\mathcal M}\approx S$ and 
\begin{equation}\label{mmspan}(\theta H^2\oplus\{0\})\vee \mathcal M=\theta H^2\oplus \mathcal H_0.\end{equation}
\end{lemma}

\begin{proof}
 Set $f=Y_0x$. Take $h\in\theta H^2$ such that 
$$ |h|\geq |f| + w^{1/2},$$ 
where $w$ is a function from  \eqref{bertak3}.
Put $$\mathcal M = \operatorname{clos}\{p(A)(h\oplus x)\ :\ p \ \text{ is an (analytic) polynomial}\}.$$
Clearly,   $\mathcal M\in\operatorname{Lat}A$, and \eqref{mmspan} is fulfilled, because $x$ is cyclic for $A_0$. 
Furthermore, set  $Y=I_{\theta H^2} \oplus Y_0$ and $\mathcal E=\operatorname{clos}Y\mathcal M$. 
Since $YA=SY$, we have $\mathcal E\in\operatorname{Lat}S$, and $Y|_{\mathcal M}A|_{\mathcal M}=S|_{\mathcal E}Y|_{\mathcal M}$. 
We will show that 
\begin{equation}\label{ysqrt}\|Yy\|\geq \frac{\|y\|}{\sqrt 2} \ \ \text{ for every } \ y\in\mathcal M.\end{equation}
Then we will obtain that $Y|_{\mathcal M}$ realizes the relation $A|_{\mathcal M}\approx S|_{\mathcal E}$. 
Since $S|_{\mathcal E}\cong S$ for every $\mathcal E\in\operatorname{Lat}S$, the lemma will be proved. 

It follows from the relation $YA=SY$ that 
\begin{align*} p(A)(h\oplus x) &= (p(S)|_{\theta H^2}h + P_{\theta H^2}p(S)|_{\mathcal K_\theta}Y_0x )\oplus p(A_0)x \\ &=
P_{\theta H^2}p(h+f)\oplus p(A_0)x\end{align*}
$$\text{and }\ Yp(A)(h\oplus x)=p(S)Y(h\oplus x)=p(h+f) \ \text{ for every polynomial }p.$$ 
Therefore, 
\begin{align*} \| p(A)&(h\oplus x)\|^2  = \|P_{\theta H^2}p(h+f)\|^2 +\| p(A_0)x\|^2 \\&
\leq \|p(h+f)\|^2 + \int_{\mathbb T}|p|^2w\text{\rm d}m 
\leq  \|p(h+f)\|^2  +\int_{\mathbb T}|p|^2|h+f|^2\text{\rm d}m \\&= 2\|p(h+f)\|^2 = 2\|Yp(A)(h\oplus x)\|^2 .\end{align*}
Thus, \eqref{ysqrt} is proved.
\end{proof}

\begin{corollary} Let $T$ be an operator from Theorem \ref{thm18}. Then there exists $\mathcal M\in\operatorname{Lat}T$ 
such that $T|_{\mathcal M}\approx S$ and \eqref{mmspan} is fulfilled (with $\theta =B$).
\end{corollary}

\begin{proof} By Lemma \ref{corg}, $Z^\ast T_B=R_0^\ast Z^\ast$. Take a vector $u\in\mathcal K_B$ such $u$ 
is cyclic for $T_B$ and set $x=Z^\ast u$. Since $\operatorname{clos}Z^\ast \mathcal K_B=\mathcal H_0$, 
$x$ is cyclic for $R_0^*$. By Lemma \ref{lembertak}, \eqref{bertak3} is fulfilled for $R_0^\ast$ and $x$. 
Thus, $T$ satisfies to the conditions of Lemma \ref{lemshift} with $\theta=B$, $A_0=R_0^\ast$ and $Y_0=W^*$. 
The conclusion of the theorem follows from Lemma \ref{lemshift}.
\end{proof}

\section{Construction of operators with the range of arbitrary codimension}

In this section,  for every cardinal number $2\leq N\leq \infty$ a  power bounded operator $T$ is constructed such that 
 $T\prec S$ and $\dim\ker T^*=N$. This is impossible for polynomially bounded operators, see Remark \ref{33} below. 
Moreover, for  the constructed operator $T$ the estimate $\|X^*h_\lambda\|\asymp\|h_\lambda\|$ is fulfilled, 
where $X$  realizes the relation $T\prec S$ and $h_\lambda$ are defined in \eqref{hlambda}. 

\begin{lemma}\label{lemker} Suppose that $\mathcal H_0$, $\mathcal E$, $\mathcal K$ are Hilbert spaces, 
$T_0\colon\mathcal H_0\to\mathcal H_0$, $A\colon\mathcal E\to\operatorname{clos}T_0\mathcal H_0$, 
$X_0\colon\mathcal H_0\to\mathcal K$, $R\colon\mathcal K\to\mathcal K$ are operators and transformations, $X_0T_0=RX_0$, 
$X_0$ is a quasiaffinity, $R$ is left invertible, $\ker A=\{0\}$ and $A\mathcal E\cap T_0\mathcal H_0=\{0\}$. 
Set \begin{gather*} T\colon \mathcal H_0\oplus\mathcal E\to\mathcal H_0\oplus\mathcal E, \ \ \ T(x\oplus e)=T_0 x + Ae 
\ \ (x\in\mathcal H_0, \ e \in\mathcal E), \\ 
 X\colon \mathcal H_0\oplus\mathcal E\to\mathcal K, \ \ \ X(x\oplus e)=X_0x+v,  \\ \text{ where } v\in\mathcal K 
 \text{ satisfies the relation } X_0Ae= Rv \ \ (x\in\mathcal H_0, \ e \in\mathcal E). \end{gather*}
Then $XT=RX$, $X$ is a quasiaffinity, $\ker T^*=\ker T_0^*\oplus\mathcal E$, and if $T_0$ is power bounded, then $T$ is power bounded.
\end{lemma}

\begin{proof} First, it needs to check that the definition of $X$ is correct, that is, 
for every $e\in\mathcal E$ there exists $v \in\mathcal K$ such that $X_0Ae= Rv$ and such $v$ is unique. 
We have 
\begin{align*}\operatorname{clos}X_0A\mathcal E & =\operatorname{clos}X_0\operatorname{clos}A\mathcal E\subset
\operatorname{clos}X_0\operatorname{clos}T_0\mathcal H_0=\operatorname{clos}X_0T_0\mathcal H_0 \\ &=
\operatorname{clos}RX_0\mathcal H_0 =\operatorname{clos}R\mathcal K =R\mathcal K,\end{align*}
the latter equality holds true due to the left invertibility of $R$. We obtain that $X_0A\mathcal E\subset R\mathcal K$. 
Thus, for every $e\in\mathcal E$ the needed $v$ exists, and the uniqueness of $v$ follows from the left invertibility of $R$ again. 
Furthermore, let $c>0$ be such that $\|Ru\|\geq c\|u\|$ for every $u\in\mathcal K$. We have 
$$ \|Xe\|=\|v\|\leq\frac{\|Rv\|}{c}=\frac{\|X_0Ae\|}{c}\leq\frac{\|X_0A\|}{c}\|e\|\ \ \ \text{ for every }e\in\mathcal E.$$
Thus, $X$ is a transformation. The equalities 
$$XT=RX, \ \  \operatorname{clos}T (\mathcal H_0\oplus\mathcal E)= \operatorname{clos}T_0\mathcal H_0, \ \text{and } \ 
\operatorname{clos}X (\mathcal H_0\oplus\mathcal E)=\mathcal K$$ 
easy follow from the definitions of $X$ and $T$. The equality for $\ker T^*$ is a consequence of the equality for the range of $T$.  
 Let $x\in\mathcal H_0$, let $e\in\mathcal E$, and let $X(x\oplus e)=0$. 
Then $$0=RX(x\oplus e)=R(X_0x+v)=RX_0x+Rv=X_0T_0x+X_0Ae=X_0(T_0x+Ae).$$ 
Since $\ker X_0=\{0\}$, we conclude that $T_0x=-Ae$. By assumption, $T_0x=-Ae=0$. Since $\ker A=\{0\}$, we obtain that 
$e=0$. From the equalities $X_0T_0=RX_0$, $\ker R=\{0\}$, and $\ker X_0=\{0\}$ we conclude that $\ker T_0=\{0\}$. Thus, $x=0$. 
We obtain that $X$ is a quasiaffinity.

Easy computation shows that $$T^n(x\oplus e)=T_0^n x + T_0^{n-1}Ae \ \ \text{ for every  } n\geq 1, \ x\in\mathcal H_0, \ e \in\mathcal E.$$
Therefore, if $T_0$ is power bounded, then $T$ is power bounded.
\end{proof}

\begin{corollary}\label{corker} For every cardinal number $1\leq N\leq\infty$ 
there exists a power bounded operator $T$ and  a quasiaffinity $X$ 
such that $XT=SX$, $\dim\ker T^*=N$, and 
\begin{equation}\label{xx}\|X^* h_\lambda\|\asymp\|h_\lambda\|=\frac{1}{(1-|\lambda|^2)^{1/2}}\ \ \ (\lambda\in\mathbb D), \end{equation} 
where  $h_\lambda$ are defined in \eqref{hlambda}. 

\end{corollary}

\begin{proof} In {\cite[Theorem 4.4]{gam}} a cyclic power bounded operator $T_0$ and a  quasiaffinity $X_0$ are constructed 
such that $X_0T_0=SX_0$, 
\begin{equation}\label{x0}\|X_0^* h_\lambda\|\asymp\|h_\lambda\|, \end{equation}
and $T_0$ is not left invertible. 
Since $T_0$ is cyclic, we have $\dim\ker T_0^*\leq 1$. Since $T_0\prec S$, we have $\dim\ker T_0^*\geq 1$. Thus, 
$\dim\ker T_0^* = 1$, and the corollary is proved for $N=1$. 

Let $N\geq 2$. Denote by $\mathcal H_0$ the space on which $T_0$ acts. 
 Set $\mathcal H_1=\operatorname{clos}T_0\mathcal H_0$. There exists a unitary transformation  
$U_1\colon\mathcal H_1\to\mathcal H_0$. Set $T_1=T_0U_1$ and consider $T_1$ as an operator on $\mathcal H_1$. 
Clearly, $T_1\mathcal H_1=T_0\mathcal H_0$. 
Since $T_0$ is not left invertible, $T_1\mathcal H_1$ is not closed. By {\cite[Theorem 3.6]{filwil}}, there exists 
a unitary operator $U_0$ on $\mathcal H_1$ such that $$T_1\mathcal H_1\cap U_0T_1\mathcal H_1=\{0\}.$$
Take a subspace $\mathcal E\subset\mathcal H_1$ such that $\dim\mathcal E=N-1$, if $N<\infty$,  or set $\mathcal E=\mathcal H_1$, if $N=\infty$. 
  Put $$A=U_0T_1|_{\mathcal E}.$$ Define $T$ as in Lemma \ref{lemker} with $R=S$. Let $X$ be a quasiaffinity from Lemma \ref{lemker}. 
By construction,  $P_{\mathcal H_0}X^*=X_0^*$. Therefore, \eqref{xx} follows from \eqref{x0}. 
Thus, $T$ and $X$ satisfy the conclusion of the corollary.
\end{proof}
 
\begin{remark}\label{33} Let $T$ be a polynomially bounded operator, and let $T\prec S$. Then $\dim\ker T^*=1$. 
Indeed, there exists a contraction $R$ such that $R\prec T$ by \cite{berpr}, and  $\dim\ker R^*=1$ by
\cite{tak}.  The range of $T$ is closed by \cite{tak} and \cite{gam1}. 
\end{remark}

\section{The essential spectrum of the operator $T_0$ from \cite{gam}}

As usually,  $\sigma(A)$, $\sigma_e(A)$ and $\sigma_p(A)$ 
denote the spectrum, the essential spectrum, and the point spectrum of an operator $A$, respectively. The following lemma is actually proved in {\cite[Theorem 1]{her}}.

\begin{lemma}\label{lemher}
Let $A$ be a cyclic operator on a Hilbert space such that $\sigma(A)\subset\operatorname{clos}\mathbb D$ and $A\prec S$. 
Then $\mathbb T\subset\sigma_e(A)$ and $\sigma_e(A)$ is connected.
\end{lemma}
\begin{proof} Since $A\prec S$, we have $\mathbb D\subset\sigma_p(A^*)$. Therefore, 
$\sigma(A)=\operatorname{clos}\mathbb D$ and $\mathbb T\subset\sigma_e(A)$ by {\cite[Theorem 0.7]{rr}}. 
Since $A$ is cyclic, $\dim\ker(A^*-\lambda I)=1$ for every $\lambda\in\mathbb D$.  
Since $A\prec S$,  $\ker(A-\lambda I)=\{0\}$ for every $\lambda\in\mathbb D$.  
We conclude that $$\mathbb D\setminus\sigma_e(A)=\{ \lambda\in\mathbb D\ : \ \operatorname{ind}(A-\lambda I)=-1\}.$$ 
By {\cite[Theorem 1]{her}}, components of $\mathbb D\setminus\sigma_e(A)$ are simple connected. 
Therefore, $\sigma_e(A)$ is connected.
\end{proof}

 We recall the detailed definition of the operator $T_0$ used in the previous section. 

The operator $T$ from  {\cite[Remark 2.2]{mt}} 
is defined as follows. 
Let $\{e_j\}_{j\geq 0}$ and $\{f_{lj}\}_{l\geq 1,j\geq 0}$ 
be orthonormal bases of 
Hilbert spaces $\mathcal E$ and $\mathcal F$, respectively. 
Put \begin{equation}\label{ttold}
\begin{gathered} Tf_{lj} = f_{l,j-1}, \ j\geq 1, \ \ Tf_{l0} = 0, \ \ l\geq 1,\\ Te_j = e_{j+1}, 
\ \ j\neq 3^k, \ k\geq 1, 
\ \ Te_{3^k} = e_{3^k+1}\! +\! f_{l,3^k},  \ \ k\geq 1, 
\  2^{l-1}\!\leq k\leq 2^l\!-1. \end{gathered}\end{equation}
It is proved in \cite{mt} that $T$ is a power bounded operator on $\mathcal E\oplus\mathcal F$.
It is easy to see that $\mathcal F\in\operatorname{Lat}T$, $T|_{\mathcal F}\cong\oplus_{l\geq 1}S^*$, and 
\begin{equation}\label{ss} P_{\mathcal E}T|_{\mathcal E}\cong S.\end{equation}
Put \begin{equation}\label{tt0}\mathcal H_0 = \bigvee_{n\geq 0}T^ne_0, \ \ 
T_0=T|_{\mathcal H_0}, 
\ \text{ and }\  X_0=P_{\mathcal E}|_{\mathcal H_0}. \end{equation} 
It is easy to see that $X_0T_0=P_{\mathcal E}T|_{\mathcal E}X_0$ and $\operatorname{clos}X_0\mathcal H_0=\mathcal E$. 
By {\cite[Lemma 4.2]{gam}}, 
\begin{equation}\label{mmold}\begin{aligned} \mathcal H_0 = \Bigl\{u\in \mathcal E\oplus\mathcal F: 
\ u=\sum_{j\geq 0}a_je_j+  \sum_{l\geq 1}\sum_{j\geq 0}
\bigl(\sum_{k\in\kappa_{lj}}a_{2\cdot 3^k+1-j}\bigr)f_{lj},  \\
   \sum_{j\geq 0}|a_j|^2 +\sum_{l\geq 1}\sum_{j\geq 0}
 \Bigl|\sum_{k\in\kappa_{lj}}a_{2\cdot 3^k+1-j}\Bigr|^2 < \infty\Bigr\}, \\
 \text{  where  } \ \kappa_{lj}=\{k\geq 1: j\leq 3^k, \  2^{l-1}\leq k\leq 2^l-1\}, \ l\geq 1, \ j\geq 0.\end{aligned}\end{equation}
It follows from \eqref{mmold} that $\ker X_0=\{0\}$. Taking into account the unitarily equivalence \eqref{ss}, we can accept  that 
$X_0$ realizes the relation $T_0\prec S$. It is proved in  {\cite[Theorem 4.4]{gam}} that \eqref{x0} is fulfilled and 
 $T_0$ is not left invertible.

\bigskip

The following lemma can be easily checked directly, therefore, its proof is omitted.
We mention only that if $n= 2\cdot 3^k+1-j$ for some $k\geq 1$ and $0\leq j\leq 3^k$, then 
such $k$ and $j$ are unique. 

\begin{lemma}\label{vzeta} Let the  Hilbert spaces $\mathcal E$ and $\mathcal F$ be defined as above, 
let $\mathcal H_0$ and $T$ be defined by \eqref{mmold} and \eqref{ttold}, and let $\zeta\in\mathbb T$. 
Define a unitary operator $V_\zeta$ on  $\mathcal E\oplus\mathcal F$ by the formulas
\begin{align*} V_\zeta e_n&=\zeta^j e_n, \begin{aligned}[t] &\ \text{ if there exist } k\geq 1 \text{ and } 0\leq j\leq 3^k \\ &\ \ \ \text{ such that } \ n= 2\cdot 3^k+1-j,\end{aligned}\\
 V_\zeta e_n&=e_n, \begin{aligned}[t]
&\ \text{ if } n\neq 2\cdot 3^k+1-j \text{ for every } k\text{ and } j  \\ &\ \ \ \text{ such that } k\geq 1\text{ and } 0\leq j\leq 3^k,\end{aligned}\\
V_\zeta f_{lj} &= \zeta^j f_{lj},  \ \ \ \ j\geq 0, \ l\geq 1.\end{align*}
Then $\mathcal F$, $\mathcal H_0\in\operatorname{Lat}V_\zeta$ and $V_\zeta \zeta T|_{\mathcal F}= T|_{\mathcal F}V_\zeta$.
\end{lemma}

\begin{theorem}\label{thmtt0} Let $T_0$ be defined by \eqref{tt0}, and let $\lambda\in\mathbb D$. If $\lambda\in\sigma_e(T_0)$, then 
$\{z :\ |z|=|\lambda|\}\subset\sigma_e(T_0).$
\end{theorem}
\begin{proof} Let $\lambda\in\sigma_e(T_0)$. Then there exists a sequence $\{u_N\}_N\subset\mathcal H_0$ such that $\|u_N\|=1$ for all $N$ and $\lim_N\|(T_0-\lambda)u_N\|=0$. 
We have $u_N=x_N\oplus y_N$, where $x_N\in\mathcal E$ and $y_N\in\mathcal F$, and
$$\|(T_0-\lambda)u_N\|^2=\|(T-\lambda)y_N+P_{\mathcal F}(T-\lambda)x_N\|^2 + \|P_{\mathcal E}(T-\lambda)x_N\|^2.$$
If $\limsup_N\|x_N\|>0$, then $\limsup_N\|P_{\mathcal E}(T-\lambda)x_N\|>0$ due to \eqref{ss}, a contradiction. Thus, 
$\lim_N\|x_N\|=0$. Consequently,  $\lim_N\|(T-\lambda)y_N\|=0$. 

Let $|z|=|\lambda|$. Then $z=\zeta\lambda$ for some $\zeta\in\mathbb T$. Let $V_\zeta$ be the operator from Lemma \ref{vzeta}. 
 We have $V_\zeta u_N\in\mathcal H_0$ and $\|V_\zeta u_N\|=1$ for all $N$. Furthermore, 
$$\|(T_0-z)V_\zeta u_N\|=\|(T-z)V_\zeta( x_N \oplus y_N)\|
\leq \|(T-z)V_\zeta x_N\|+ \|(T-z) V_\zeta y_N\|.$$
We have   $\lim_N\|(T-z)V_\zeta x_N\|=0$.  By Lemma \ref{vzeta},  
$$ (T-z) V_\zeta y_N = V_\zeta(\zeta T-z) y_N =V_\zeta \zeta (T-\lambda)y_N \ \ \text{ for every } N.$$
We have   $\| (T-z) V_\zeta y_N\|=\|(T-\lambda)y_N\|\to 0.$ Thus, the  sequence $\{V_\zeta u_N\}_N$ 
shows that $z\in\sigma_e(T_0)$.
\end{proof}

\begin{corollary} Let $T_0$ be defined by \eqref{tt0}. Then $\sigma_e(T_0)=\operatorname{clos}\mathbb D$.
\end{corollary}

\begin{proof} By  Lemma \ref{lemher}, $\sigma_e(T_0)$ is connected. 
If $0\neq\lambda\in\mathbb D$ and $\lambda\not \in\sigma_e(T_0)$, then, by Theorem \ref{thmtt0},  
$$\{z :\ |z|=|\lambda|\}\subset\mathbb D\setminus\sigma_e(T_0).$$ Since the set $\{z :\ |z|=|\lambda|\}$ is compact and the set 
$\mathbb D\setminus\sigma_e(T_0)$ is open, there exists $\varepsilon>0$ such that 
$$\{z :\ |\lambda|-\varepsilon<|z|< |\lambda|+\varepsilon\}\subset\mathbb D\setminus\sigma_e(T_0).$$ 
By Lemma \ref{lemher},  $\mathbb T\subset\sigma_e(T_0)$. By {\cite[Theorem 4.4]{gam}}, $0\in\sigma_e(T_0)$.  
Therefore,  $\sigma_e(T_0)$ can not be connected, a contradiction. 
\end{proof}

\end{document}